\documentclass{aims}
\usepackage{amsmath}
  \usepackage{paralist}
  \usepackage{graphics} %% add this and next lines if pictures should be in esp format
  \usepackage{epsfig} %For pictures: screened artwork should be set up with an 85 or 100 line screen
\usepackage{graphicx}  \usepackage{epstopdf}%This is to transfer .eps figure to .pdf figure; please compile your paper using PDFLeTex or PDFTeXify.
 \usepackage[colorlinks=true]{hyperref}
\usepackage[pagewise]{lineno}
\hypersetup{urlcolor=blue, citecolor=red}

\def\currentvolume{X}
 \def\currentissue{X}
  \def\currentyear{200X}
   \def\currentmonth{XX}
    \def\ppages{X--XX}
     \def\DOI{10.3934/xx.xx.xx.xx}

\newtheorem{theorem}{Theorem}[section]

\newtheorem{lemma}[theorem]{Lemma}

\theoremstyle{definition}
\newtheorem{definition}[theorem]{Definition}
\newtheorem{remark}{Remark}

\newtheorem{algorithm}{Algorithm}

\title[Projection methods for solving split equilibrium problems] %Use the shortened version of the full title
      {Projection methods for solving split equilibrium problems}

\author[Dang Van Hieu]{}


 \email{dangvanhieu@tdtu.edu.vn}
\thanks{$^*$ Corresponding author: dangvanhieu@tdtu.edu.vn}

\begin{document}
\maketitle

% Enter the first author's name and address:
\centerline{\scshape Dang Van Hieu$^*$}
\medskip
{\footnotesize
% please put the address of the first author
 \centerline{Applied Analysis Research Group, Faculty of Mathematics and Statistics,}
   \centerline{Ton Duc Thang University, Ho Chi Minh City, Vietnam}
  % \centerline{ Springfield, MO 65801-2604, USA}
} % Do not forget to end the {\footnotesize by the sign }

\medskip
% 
% \centerline{\scshape First-name2 last-name2 and First-name3
% last-name3}
% \medskip
% {\footnotesize
%  % please put the address of the second  and third author
%  \centerline{ First line of the address of the second author}
%    \centerline{Other lines}
%    \centerline{Springfield, MO 65810, USA}
% }
% 
\bigskip

% The name of the associate editor will be entered by an editorial staff
% "Communicated by the associate editor name" is not needed for special issue.
 \centerline{(Communicated by the associate editor name)}

%The abstract of your paper
\begin{abstract}
The paper considers a split inverse problem involving component equilibrium problems in Hilbert spaces. This problem therefore 
is called the split equilibrium problem (SEP). It is known that almost solution methods for solving problem (SEP) are designed from two 
fundamental methods as the proximal point method and the extended extragradient method (or the two-step proximal-like method). Unlike 
previous results, in this paper we introduce a new algorithm, which is only based on the projection method, for finding solution approximations of problem (SEP), 
and then establish that the resulting algorithm is weakly convergent under mild conditions. Several of numerical results are reported to 
illustrate the convergence of the proposed algorithm and also to compare with others.
\end{abstract}

%%%%%%%%%%%%%%%%%%%%%%%%%%%%%%%%%%
\section{Introduction}\label{intro}
The split feasibility problem \cite{CE1994} consists of finding a point in a closed convex subset of a space such that its image under a bounded linear 
operator belongs to a closed convex subset of another space. This problem has received a lot of attention because of its applications in signal processing, 
specifically in phase retrieval and other image restoration problems, see, e.g., \cite{H1989,S1987}. After that, it was found that the split feasibility problem can be used to 
model the intensity-modulated radiation therapy \cite{CS2008}, and many other fields \cite{B2002,B2004,CBMT2006}. That is also the reason to explain 
why in recent years many split-like problems have been widely and intensively studied, for instance, the split fixed point problem, the split optimization problem 
and the split variational inequality problem \cite{CGR2012,M2011} and others \cite{CWWW2015,M2013,M2013a,YMH2016}. Mathematically, these problems can 
be modelled in a common form, and so-called the split inverse problem (SIP), see in \cite[Sect. 2]{CGR2012}, in which there are a bounded linear operator $A$ 
from a space $X$ to another space $Y$ and two inverse problems IP1 and IP2 installed in $X$ and $Y$, respectively. More precisely, the problem (SIP) is of the form,
\begin{equation*}\label{SIP}
\begin{cases}
\mbox{find a point }~x^*\in X~\mbox{that solves IP1}\\
\mbox{such that}\\
\mbox{the point }~y^*=Ax^*\in Y~\mbox{solves IP2}.
\end{cases}
\eqno{\rm (SIP)}
\end{equation*}
Based on this general model, we can consider various types of split problems, even extend them to split equality-like problems.
Recall that the equilibrium problem \cite{BO1994,CH2005,MO1992} for a bifunction $f:C\times C \to \Re$ is to find a point $x^*\in C$ such that 
\begin{equation*}
f(x^*,y)\ge 0,~\forall y\in C,
\eqno{\rm (EP)}
\end{equation*}
where $C$ is a nonempty closed convex subset of a real Hilbert space $H$. Let us denote by $EP(f,C)$ the solution set of the problem (EP). It was well known that problem (EP) unifies in a simple form many 
mathematical models such as the variational inequalities, the fixed point problems, the optimization problems and the Nash equilibrium 
problems, see, e.g., \cite{BO1994,H2017OPTL,HM2017,HS18,H18,HCB18,MO1992}. It is here natural in this framework to study problem (SIP) when IP1 and IP2 are equilibrium 
problems to get the so-called split equilibrium problem (SEP). The problem of this form has also been considered recently in \cite{H2012,H2017GCOM,M2011}. 
More precisely, the problem (SEP) is stated as follows:\\

\vspace{0.1cm}
\noindent\textbf{Problem (SEP)}: \textit{Let $H_1,~H_2$ be two real Hilbert spaces and $C,~Q$ be two nonempty closed convex subsets of $H_1,~H_2$, respectively. 
Let $A:H_1\to H_2$ be a bounded linear operator. Let $f:C\times C\to \Re$ and $F:Q\times Q\to \Re$ be two bifunctions with 
$f(x,x)=0$ for all $x\in C$ and $F(u,u)=0$ for all $u\in Q$. The problem (SEP) is:
\begin{equation*}\label{SEP}
\begin{cases}
\mbox{Find}~x^*\in C~\mbox{such that} ~f(x^*,y)\ge 0,~\forall y\in C,\\
\mbox{and}~u^*=Ax^*\in Q~\mbox{solves} ~F(u^*,v)\ge 0,~\forall v\in Q.
\end{cases}
\eqno{\rm (SEP)}
\end{equation*}}\\

\vspace{0.1cm}
\noindent Let $\Omega$ denote the solution set of problem (SEP), i.e., 
$$\Omega=\left\{x^*\in C: f(x^*,y)\ge 0~\mbox{and}~F(Ax^*,v)\ge 0,~\forall y\in C,~\forall v\in Q\right\}.$$ 
Several methods for solving problem (SEP) can be found, for instance, in 
\cite{DMSK2015,DKK2014,H2012,H2016,KR2013,M2011}. As far as we know, almost solution methods for solving problem 
(SEP) are based on the proximal point method \cite{M1999} which consists of computing the resolvents $T^f_r$ and $T^F_s$ of 
bifunctions $f,~F$ with some $r,~s>0$. Recall that the resolvent \cite{CH2005} of a bifunction $f:C\times C\to \Re$ with some $r>0$ is defined by 
\begin{equation}\label{T}
T^f_r(x)=\left\{z\in C: f(z,y)+\frac{1}{r}\left\langle y-z,z-x\right\rangle\ge 0,~\forall y\in C\right\}.
\end{equation} 
Recently, the author of \cite{H2016} has introduced the extragradient-proximal method \cite[Corollary 1]{H2016}, for solving problem (SEP), 
which combines three methods including the proximal point method \cite{M1999}, the extended extragradient method \cite{FA1997,QMH2008} 
and the projection method. Also, recall 
here that the extended extragradient method \cite{FA1997,QMH2008} involves the computations of the following two optimization programs, 
for each $x\in C$,
\begin{equation}\label{EGM}
\begin{cases}
y=\arg\min\{ \lambda f(x, t) +\frac{1}{2}||x-t||^2:t\in C\},\\
z=\arg\min\{ \lambda f(y, t) +\frac{1}{2}||x-t||^2:t\in C\},
\end{cases}
\end{equation}
where some $\lambda>0$. It seems that the extended extragradient method (\ref{EGM}) can be easier to compute numerically than the proximal-point 
method, that is $T^f_r(x)$, which comes from the nonlinear inequality in (\ref{T}). However, the solving of two optimization programs in (\ref{EGM}) can 
be still costly if the bifunction $f$ and the feasible set $C$ have complex structures. Very recently, the author of \cite{H2017GCOM} 
has presented a new algorithm (see, \cite[Algorithm 3.1]{H2017GCOM}) for solving problem (SEP), namely the projected subgradient - proximal method (PSPM).  
More precisely, the PSPM is designed as follows:\\
%%%%%%%%%%%%%%%%%%%%%%%%%%%%%%%%%%%%%%%%

\vspace{0.1cm}
%%%%%%%%%%%%%%%%%%%%%%%%%%%%%%%%%%%%%%%
\noindent \textbf{Algorithm (PSPM)}\\

\vspace{0.1cm}
%%%%%%%%%%%%%%%%%%%%%%%%%%%%%%%%%%%%%%%%
\noindent \textbf{Initialization:} Choose $x_0\in C$ and the parameter sequences $\left\{\rho_n\right\}$, $\left\{\beta_n\right\}$, $\left\{\epsilon_n\right\}$, 
$\left\{r_n\right\}$, $\left\{\mu_n\right\}$ such that\\[.1in]

\indent (i) $\rho_n\ge \rho>0$, $\beta_n>0$, $\epsilon_n\ge 0$, $r_n\ge r>0$.\\

\indent (ii) $\sum\limits_{n\ge 1}\frac{\beta_n}{\rho_n}=+\infty$,\qquad  $\sum\limits_{n\ge 1}\frac{\beta_n\epsilon_n}{\rho_n}<+\infty$,\qquad 
$\sum\limits_{n\ge 1}\beta_n^2<+\infty$.\\

\indent (iii) $0<a\le \mu_n\le b<\frac{2}{||A||^2}$. \\[.1in]

%%%%%%%%%%%%%%%%%%%%%%%%%%%%%%%%%%%%%%%%%%%%%%%
\noindent\textbf{Iterative Steps:} Assume that $x_n\in C$ is known, calculate $x_{n+1}$ as follows:\\[.1in]

\indent \textbf{Step 1.} Select $w_n\in \partial_{\epsilon_n}f(x_n,.)(x_n)$, and compute 
$$\gamma_n=\max\left\{\rho_n,||w_n||\right\},\qquad \alpha_n=\frac{\beta_n}{\gamma_n},\qquad y_n=P_C(x_n-\alpha_n w_n).$$
\indent \textbf{Step 2.} Compute $ x_{n+1}=P_C\left(y_n-\mu_n A^*(I-T_{r_n}^F)Ay_n\right)$.\\
\vspace{0.1cm}

%%%%%%%%%%%%%%%%%%%%%%%%%%%%%%%%%%%%%%%%%%%%%%%
A modification of PSPM was also introduced in \cite[Algorithm 4.1]{H2017GCOM} where the prior knowledge of the norm of operator $A$ is not 
necessary. As be seen, the PSPM is constructed around the two methods, namely the projection method and the proximal point method \cite{M1999} 
(i.e., using the resolvent mapping $T_{r_n}^F$ of bifunction $F$). Finding a value of resolvent mapping in general is not easy. Then, the introduction of a computable 
and effective algorithm is necessary.\\[.1in]
%%%%%%%%%%%%%%%%%%%%%%%%%%%%%%
\textit{\textbf{Our concern now is the following}: Can we construct an algorithm for solving problem (SEP) which only uses the projection method?}\\[.1in]
%%%%%%%%%%%%%%%%%%%%%%%%%%%%%%
In this paper, as a continuity of the results in \cite{H2016,H2017GCOM}, we introduce a different algorithm for approximating solutions of problem 
(SEP) to answer the aforementioned question. Unlike the existing results, we only use the projection methods to design the algorithm. Theorem of weak 
convergence is proved under mild conditions. For further purpose, we consider some simple examples to demonstrate that several considered conditions 
are necessary in the formulation of theorem of convergence. The resulting algorithm is also extended to solve other related form-like problems. Finally, we 
perform several experiments to illustrate the numerical behavior of the new algorithm and aslo to compare it with others. The analyses 
in this paper are based on the ones in the recent work \cite{SS2011}. In this direction, a special case of problem (SEP) has been studied by the authors 
in \cite{AM14}. A generalization of problem (SEP) with fixed point problems and the methods of proximal-extragradient form can be found in \cite{DSA17}.\\[.1in]
%%%%%%%%%%%%%%%%%%%%%%%%%%%%%%
An outline of this paper is as follows: In Sect. \ref{pre}, we recall some definitions and preliminary results for further use. Sect. \ref{main} 
deals with proposing the algorithm and analyzing its convergence. Some further remarks are presented in Sect. \ref{remarks} to justify 
the introduction of the assumptions in the convergence theorem. Sect. \ref{extension} introduces an extension of the resulting algorithm 
to the split common equilibrium problem. In Sect. \ref{example} we perform several numerical experiments to illustrate the computational 
efficiency of the proposed algorithm and also to compare it with others.
%%%%%%%%%%%%%%%%%%%%%%%%%%%%%%%%%%%%%%%%%%%%%%%%%
\section{Preliminaries}\label{pre}
%%%%%%%%%%%%%%%%%%%%%%%%%%%%%%%%%%%%%%%%%%%%%%%
Let $C$ be a nonempty closed convex subset of a real Hilbert space $H$. 
The metric projection $P_C:H\to C$ is defined by
$$
P_C(x)=\arg\min\left\{\left\|y-x\right\|:y\in C\right\}.
$$
Since $C$ is nonempty, closed and convex, $P_C(x)$ exists and is unique. From the definition of the metric projection, 
it is easy to show that $P_C$ has the following property.
\begin{lemma}\label{lemPC}\cite{GR1984}
{\rm (i)} $\left\langle P_C(x)-P_C(y),x-y \right\rangle \ge \left\|P_C (x)-P_C (y)\right\|^2,~\forall x,y\in H.$\\[.1in]
{\rm (ii)} $\left\|x-P_C (y)\right\|^2+\left\|P_C (y)-y\right\|^2\le \left\|x-y\right\|^2, \forall x\in C, y\in H.$\\[.1in]
{\rm (iii)} $z=P_C (x) \Leftrightarrow \left\langle x-z,y-z \right\rangle \le 0,\quad \forall y\in C.$
\end{lemma}
%%%%%%%%%%%%%%%%%%%%%%%%%%%%%%%%%%%
Now, we recall some concepts of monotonicity of a bifunction, see, e.g., \cite{BO1994,MO1992}. 
\begin{definition} A bifunction $f:C\times C\to \Re$ is said to be:\\[.1in]
%%%%%%%%%%%%%%%%%%%%%%%%%%%%%%%%
(i) strongly monotone on $C$ if there exists a constant $\gamma>0$ such that
$$ f(x,y)+f(y,x)\le -\gamma ||x-y||^2,~\forall x,y\in C; $$
(ii) monotone on $C$ if 
$$ f(x,y)+f(y,x)\le 0,~\forall x,y\in C; $$
(iii) pseudomonotone on $C$ if 
$$ f(x,y)\ge 0 \Longrightarrow f(y,x)\le 0,~\forall x,y\in C.$$
(iv) strongly pseudomonotone on $C$ if there exists a constant $\gamma>0$ such that
$$ f(x,y)\ge 0 \Longrightarrow f(y,x)\le -\gamma ||x-y||^2,~\forall x,y\in C.$$
\end{definition}
From the above definitions, it is clear that the following implications hold,
$$ {\rm (i)}\Longrightarrow {\rm (ii)}\Longrightarrow {\rm (iii)}~{\rm and}~{\rm (i)}\Longrightarrow {\rm (iv)}\Longrightarrow {\rm (iii)}.$$
The converses in general are not true. Recall that a function $\varphi:C\to \Re$ is said to be convex on $C$ if for all $x,y\in C$ and $t\in [0,1]$,
$$ \varphi(tx+(1-t)y)\le t\varphi(x)+(1-t)\varphi (y).$$
The subdifferential of $\varphi$ at $x\in C$ is defined by
$$\partial \varphi (x)= \left\{w\in H:\varphi(y)-\varphi(x)\ge \left\langle w,y-x\right\rangle,~\forall y\in C\right\}.$$
An enlargement of the subdifferential is the $\epsilon$-subdifferential. The $\epsilon$-subdifferential of $\varphi$ at $x\in C$ is defined by
$$\partial_{\epsilon} \varphi (x)= \left\{w\in H:\varphi(y)-\varphi(x)+\epsilon\ge \left\langle w,y-x\right\rangle,~\forall y\in C\right\}. $$ 
It is clear that the $0$-subdifferential coincides with the subdifferential. Let $f:C\times C\to \Re$ be a bifunction. Throughout this paper,  
$\partial_{\epsilon} f(x,.)(x)$ is called the $\epsilon$-diagonal subdifferential of $f$ at $x\in C$.  \\[.1in]
%%%%%%%%%%%%%%%%%%%%%%%%%%%%%%%%%%%%%%%%%%%
We need the following technical lemma to prove the convergence of the proposed algorithms.
\begin{lemma}\cite{X2004}\label{lem.tech1}
Let $\left\{\nu_n\right\}$ and $\left\{\delta_n\right\}$ be two sequences of positive real numbers such that 
$$\nu_{n+1}\le \nu_n+\delta_n,~\forall n\ge 1,$$
with $\sum_{n\ge 1}\delta_n<+\infty$. Then the sequence $\left\{\nu_n\right\}$ is convergent.
\end{lemma}
%%%%%%%%%%%%%%%%%%%%%%%%%%%%%%%%%%%%%
\section{Algorithm and convergence}\label{main}
%%%%%%%%%%%%%%%%%%%%%%%%%%%%%%%%%%%%%%%%
% \setcounter{lemma}{0}
% \setcounter{theorem}{0}
% \setcounter{remark}{0}
% \setcounter{corollary}{0}
% \setcounter{algorithm}{0}
%%%%%%%%%%%%%%%%%%%%%%%%%%%%%
In this section, we introduce a new algorithm for approximating solutions of problem (SEP). For designing our algorithm, throughout the paper, 
we take four non-negative parameter sequences $\left\{\rho_n\right\}$, $\left\{\beta_n\right\}$, $\left\{\epsilon_n\right\}$, and $\left\{\mu_n\right\}$ 
satisfying the following conditions.\\[.1in]
C1. $\rho_n\ge \rho>0$, $\epsilon_n \ge 0$, $\beta_n>0$.\\[.1in]
C2. $\sum\limits_{n\ge 1}\frac{\beta_n}{\rho_n}=+\infty$,\qquad  $\sum\limits_{n\ge 1}\frac{\beta_n\epsilon_n}{\rho_n}<+\infty$,\qquad 
$\sum\limits_{n\ge 1}\beta_n^2<+\infty$.\\[.1in]
C3. $0<a\le \mu_n\le \frac{1}{||A||^2}$. \\[.1in]
% %%%%%%%%%%%%%%%%%%%%%%%%%%%%%%%%%%%%%
The following is the algorithm in details.\\
%%%%%%%%%%%%%%%%%%%%%%%%%%%%%%%%%%%%%%
\noindent\rule{12.6cm}{0.4pt} 
\begin{algorithm}[Projection Method for SEPs].\label{alg1}\\
\noindent\rule{12.6cm}{0.4pt}\\
\textbf{Initialization:} Choose $x_0\in C$ and parameter sequences $\left\{\rho_n\right\}$, $\left\{\beta_n\right\}$, $\left\{\epsilon_n\right\}$, 
$\left\{\mu_n\right\}$ such that conditions C1-C3 above hold.\\
\noindent\rule{12.6cm}{0.4pt}\\
\textbf{Iterative Steps:} Assume that $x_n\in C$ is known, calculate $x_{n+1}$ as follows:

\textbf{Step 1.} Select $w_{n}\in \partial_{\epsilon_n}F(u_n,.)(u_n)$ where $u_n=P_Q(Ax_n)$, and compute 
$$\gamma_n=\frac{\beta_n}{\max\left\{\rho_n,||w_{n}||\right\}},\qquad y_n=P_Q(u_n-\gamma_n w_{n}).$$

\textbf{Step 2.} Compute $ z_{n}=P_C\left(x_n+\mu_n A^*(y_n-Ax_n)\right)$.

\textbf{Step 3.} Select $g_n\in \partial_{\epsilon_n}f(z_n,.)(z_n)$ and compute 
$$\alpha_{n}=\frac{\beta_n}{\max\left\{\rho_n,||g_n||\right\}},\qquad x_{n+1}=P_C(z_n-\alpha_{n} g_n).$$

Set $n:=n+1$ and go back to \textbf{Step 1}.
\end{algorithm}
\noindent\rule{12.6cm}{0.4pt}
%%%%%%%%%%%%%%%%%%%%%%%%%%%%%%%%%%%%
\begin{remark}
Remark that since $\rho_n\ge \rho>0$, Algorithm \ref{alg1} is well defined. In view of the existing methods in 
\cite{CWWW2015,DMSK2015,DKK2014,DSA17,H2016,KR2013,M2011}, we see that they are almost designed 
in combining two methods: the proximal point method (that is to compute the resolvent of a bifunction \cite{M1999}) or the extended extragradient method 
(or the two-step proximal-like method \cite{FA1997,QMH2008}). Algorithm \ref{alg1} is close to the methods in 
\cite{H2017GCOM}. However, it only uses the projection method to design it and without the resolvent mapping $T^F_r$ of $F$ as in \cite{H2017GCOM}. 
As mentioned above, the using of resolvent mapping can be time-consuming in numerical computation when the bifunctions and the feasible sets have 
complicated structures.
\end{remark}
%%%%%%%%%%%%%%%%%%%%%%%%%%%%%%%%%%%%%%%%%%%%%%%
\begin{remark}
From the condition $\sum\limits_{n = 1}^\infty \beta_n^2<+\infty$ in (C2), we see that $\lim\limits_{n\to\infty}\beta_n=0$. 
This implies that the sequences of stepsizes $\left\{\gamma_n\right\}$ and $\left\{\alpha_n\right\}$ in Algorithm \ref{alg1} are decreasing. In general 
this strategy is not good. However, this assumption allows the algorithm to work without imposing the Lipschitz-type condition on the bifunction.
So doing, the stepsizes are suitably updated at each iteration and are independent on the Lipschitz-type constants.
\end{remark}

\noindent In order to establish the convergence of Algorithm \ref{alg1}, we assume that the bifunction $f:C\times C\to \Re$ satisfies the following conditions. \\[.1in]
%%%%%%%%%%%%%%%%%%%%%%%%%%%%%%%%%
A1. $f$ is pseudomonotone and $f(x,x)=0,~\forall x\in C$;\\[.1in]
A2. $f(x,.)$ is convex and lower semicontinuous on $C$ and $f(.,y)$ is weakly upper semicontinuous on $C$;\\[.1in]
A3. The $\epsilon$-diagonal subdifferential of $f$ is bounded on each bounded subset of $C$;\\[.1in]
A4. $f$ satisfies the following paramonotone condition
$$x\in EP(f,C),~y\in C,~f(y,x)=0\Longrightarrow y\in EP(f,C).$$
%%%%%%%%%%%%%%%%%%%%%%%%%%%%%%%%%%%%%%%%%
In addition, the bifunction $F:Q\times Q\to \Re$ is also assumed to satisfy the properties A1 - A4 above, but on the feasible set $Q$. 
Several remarks on these assumptions will be presented in the next section where it is seen that paramonotone condition A4 is necessary to establish the 
convergence of the algorithm. Under conditions A1 and A2, the two sets $EP(f,C)$ and $EP(F,Q)$ are closed and convex. Thus, 
since $A$ is linear, the solution set $\Omega$ of problem (SEP) is also closed and convex. In this paper, $\Omega$ is assumed to be nonempty, and 
so the projection $P_\Omega(u)$ is well defined for each point $u\in H_1$. To investigate the asymptotic behavior of the sequence $\left\{x_n\right\}$ 
generated by Algorithm \ref{alg1}, we need the following lemmas.
%%%%%%%%%%%%%%%%%%%%%%%%%%%%%%%%%%%%%%

\begin{lemma}\label{lem0}
Let $x^*\in EP(f,C)$. Then we have the following estimate for each $n\ge 0$,
\begin{equation*}
||x_{n+1}-x^*||^2\le ||z_n-x^*||^2+2\alpha_n f(z_n,x^*)-||x_{n+1}-z_n||^2+\delta_n,
\end{equation*}
where $\delta_n=\frac{2\beta_n\epsilon_n}{\rho_n}+2\beta_n^2$.
\end{lemma}
\begin{proof}
See, e.g., \cite[inequality (8)]{H2017GCOM}. 
\end{proof}
As in Lemma \ref{lem0} but with $Ax^*\in EP(F,Q)$, from the definition of $y_n$, we also have the following estimate for each $n\ge 0$,
\begin{equation*}
||y_n-Ax^*||^2\le ||u_n-Ax^*||^2+2\gamma_n F(u_n,Ax^*)-||y_n-u_n||^2+\delta_n,
\end{equation*}
where $\delta_n$ is defined as in Lemma \ref{lem0}. Thus 
\begin{equation}\label{eq:1}
||y_n-Ax^*||^2\le ||u_n-Ax^*||^2+2\gamma_n F(u_n,Ax^*)+\delta_n.
\end{equation}
%%%%%%%%%%%%%%%%%%%%%%%%%%%%%%%%%%%%%%
\begin{lemma}\label{lem1} Let $x^*\in \Omega$. Then the following inequality holds for each $n\ge 0$,
\begin{eqnarray*}
||x_{n+1}-x^*||^2&\le& ||x_n-x^*||^2-\mu_n ||u_n-Ax_n||^2-||x_{n+1}-z_n||^2\nonumber\\
&&+2\mu_n \gamma_n F(u_n,Ax^*)+2\alpha_n f(z_n,x^*)+(1+\mu_n)\delta_n.
\end{eqnarray*}
\end{lemma}
%%%%%%%%%%%%%%%%%%%%%%%%%%%%%%%%
\begin{proof}
From the definition of $z_n$ and the nonexpansiveness of $P_C$, 
we obtain 
\begin{eqnarray}
||z_n-x^*||^2&=&||P_C\left(x_n+\mu_n A^*(y_n-Ax_n)\right)-P_C(x^*)||^2\nonumber\\ 
&\le &||x_n+\mu_n A^*(y_n-Ax_n)-x^*||^2\nonumber\\
&=&||x_n-x^*||^2+\mu_n^2||A^*(y_n-Ax_n)||^2+2\mu_n \left\langle x_n-x^*,A^*(y_n-Ax_n)\right\rangle\nonumber\\ 
&\le &||x_n-x^*||^2+\mu_n^2||A||^2||y_n-Ax_n||^2+2\mu_n \left\langle A(x_n-x^*),y_n-Ax_n\right\rangle.\nonumber\\\label{eq:2}
\end{eqnarray}
Now, we estimate the term $\left\langle A(x_n-x^*),y_n-Ax_n\right\rangle$ in inequality (\ref{eq:2}). Since $P_Q$ 
is firmly nonexpansive, we obtain 
\begin{eqnarray*}
2||u_n-Ax^*||^2&=&2||P_Q(Ax_n)-P_Q(Ax^*)||^2\nonumber\\ 
&\le &2\left\langle P_Q(Ax_n)-P_Q(Ax^*),Ax_n-Ax^*\right\rangle\nonumber\\
&=&2\left\langle u_n-Ax^*,Ax_n-Ax^*\right\rangle\nonumber\\
&=&||u_n-Ax^*||^2+||Ax_n-Ax^*||^2-||u_n-Ax_n||^2.
\end{eqnarray*}
Thus
\begin{equation}\label{eq:3}
||u_n-Ax^*||^2\le ||Ax_n-Ax^*||^2-||u_n-Ax_n||^2.
\end{equation}
Combining relations (\ref{eq:1}) and (\ref{eq:3}) we obtain 
\begin{equation}\label{eq:5}
||y_n-Ax^*||^2\le ||Ax_n-Ax^*||^2-||u_n-Ax_n||^2+2\gamma_n F(u_n,Ax^*)+\delta_n,
\end{equation}
which, together with the following equality
$$
2 \left\langle A(x_n-x^*),y_n-Ax_n\right\rangle=||y_n-Ax^*||^2-||Ax_n-Ax^*||^2-||y_n-Ax_n||^2,
$$
implies that 
\begin{equation}\label{eq:6}
2 \left\langle A(x_n-x^*),y_n-Ax_n\right\rangle \le -||u_n-Ax_n||^2-||y_n-Ax_n||^2+2\gamma_n F(u_n,Ax^*)+\delta_n.
\end{equation}
Combining relations (\ref{eq:2}) and (\ref{eq:6}), we get 
\begin{eqnarray}
||z_n-x^*||^2&\le&||x_n-x^*||^2-\mu_n(1-\mu_n||A||^2)||y_n-Ax_n||^2-\mu_n ||u_n-Ax_n||^2 \nonumber\\
&&+2\mu_n \gamma_n F(u_n,Ax^*)+\mu_n \delta_n\nonumber\\ 
&\le&||x_n-x^*||^2-\mu_n ||u_n-Ax_n||^2+2\mu_n \gamma_n F(u_n,Ax^*)+\mu_n \delta_n,\label{eq:7}
\end{eqnarray}
where the last inequality follows from assumption C3 that $\mu_n(1-\mu_n||A||^2)\ge 0$. This together with Lemma \ref{lem0} 
implies the desired conclusion. Lemma \ref{lem1} is proved.
\end{proof}
% %%%%%%%%%%%%%%%%%%%%%%%%%%%%%%%%%%%%%
\begin{lemma}\label{lem2} Let $\left\{x_n\right\}$ be the sequence generated by Algorithm \ref{alg1}. Then the following properties are satisfied.\\[.1in]
{\rm (i)} The sequence $\left\{||x_n-x^*||^2\right\}$ is convergent for each $x^*\in \Omega$, and the sequence $\left\{x_n\right\}$ is bounded.\\[.1in]
{\rm (ii)} $ \lim_{n\to\infty} ||x_{n+1}-z_n||^2=\lim_{n\to\infty} ||u_n-Ax_n||^2=0$, and the sequences $\left\{z_n\right\}$, $\left\{u_n\right\}$ are bounded.\\[.1in]
{\rm (iii)} $\lim_{n\to\infty}\sup f(x_n,x^*)=\lim_{n\to\infty}\sup F(u_n,Ax^*)=0$ for each $x^*\in \Omega$.
\end{lemma}
%%%%%%%%%%%%%%%%%%%%%%%%%%%%%%%%
\begin{proof}
(i) Since $x^*\in EP(f,C)$ and $x_n\in C$, we have $f(x^*,x_n)\ge 0$. Then $f(x_n,x^*)\le 0$ by the pseudomonotonicity of $f$. Similarly, 
from $u_n\in Q$ and $Ax^*\in EP(F,Q)$, we also have $F(u_n,Ax^*)\le 0$. These together with Lemma \ref{lem1}, $\mu_n>0$, $\alpha_n>0$, 
$\gamma_n>0$ imply that 
\begin{equation}\label{eq7a}
||x_{n+1}-x^*||^2\le||x_n-x^*||^2+(1+\mu_n)\delta_n.
\end{equation}
Using Lemma \ref{lem.tech1} and the fact that $\sum_{n\ge 1}(1+\mu_n)\delta_n<+\infty$, it follows that the sequence $\left\{||x_n-x^*||^2\right\}$ 
converges and thus that $\left\{x_n\right\}$ is bounded.\\[.1in]
%%%%%%%%%%%%%%%%%%%%%%%%%%%%%%
(ii) For the sake of simplicity, we set
$$ M_n=\mu_n ||u_n-Ax_n||^2+||x_{n+1}-z_n||^2\ge 0, $$
$$ N_n=- 2\mu_n \gamma_n F(u_n,Ax^*)-2\alpha_n f(z_n,x^*)\ge 0.$$
Thus, the inequality in Lemma \ref{lem1} can be shortly rewritten as 
$$M_n+N_n\le||x_n-x^*||^2-||x_{n+1}-x^*||^2+(1+\mu_n)\delta_n.$$
Let $N\ge 1$ be a fixed integer number. Summing up these inequalities for $n=1,2,\ldots,N$, we obtain
$$0\le \sum_{n=1}^N M_n+\sum_{n=1}^N N_n\le||x_1-x^*||^2-||x_{N+1}-x^*||^2+\sum_{n=1}^N (1+\mu_n)\delta_n. $$
This is true for all $N\ge 1$. Passing to the limit in the last inequality as $N\to\infty$, and using Lemma \ref{lem2}(i) and the fact that
$\sum_{n\ge 1}(1+\mu_n)\delta_n<+\infty$ , we obtain
$${\rm (S1)}\quad\sum_{n=1}^{\infty} M_n<+\infty,\qquad {\rm (S2)}\quad\sum_{n=1}^{\infty}N_n<+\infty.$$
From (S1) and the definition of $M_n$, we obtain 
\begin{equation}\label{eq:8}
\lim_{n\to\infty}||x_{n+1}-z_n||^2=0
\end{equation}
and $\lim_{n\to\infty} \mu_n ||u_n-Ax_n||^2=0.$ This together with the hypothesis $\mu_n\ge a>0$ implies that 
\begin{equation}\label{eq:9}
\lim_{n\to\infty} ||u_n-Ax_n||^2 =0.
\end{equation}
Thus, from the boundedness of $\left\{x_n\right\}$ and the linearity of operator $A$, we also obtain that the two sequences $\left\{z_n\right\}$, $\left\{u_n\right\}$ 
are bounded.\\[.1in]
%%%%%%%%%%%%%%%%%%%%%%%%%%%%%%
(iii) From (S2), the definition of $N_n$, and the facts $- \mu_n \gamma_n F(u_n,Ax^*)\ge 0$ and $-2\alpha_n f(z_n,x^*)\ge 0$ for all $n\ge 0$, we obtain 
 $${\rm (S3)}\quad \sum_{n=1}^{\infty} \alpha_n [-f(z_n,x^*)]<+\infty,\qquad {\rm (S4)}\quad \sum_{n=1}^{\infty}\mu_n \gamma_n [-F(u_n,Ax^*)]<+\infty.$$
Hence from (S4) and hypothesis C3, we can deduce that 
$$ {\rm (S5)}\quad \sum_{n=1}^{\infty}\gamma_n [-F(u_n,Ax^*)]<+\infty. $$
On the other hand, since $\left\{z_n\right\}$ is bounded, it follows from assumption A3 that $\left\{g_n\right\}$ is also bounded. Thus, there exists $L\ge \rho>0$ 
such that $||g_n||\le L$, and from the definition of $\alpha_n$ and C1, we can write
$$ \alpha_n=\frac{\beta_n}{\max\left\{\rho_n,||w_n||\right\}}=\frac{\beta_n}{\rho_n\max\left\{1,||w_n||/\rho_n\right\}}
\ge \frac{\beta_n}{\rho_n} \frac{\rho}{L}.$$
This together with (S3) and $\frac{\rho}{L}>0$ implies that
$$ \sum_{n=1}^{\infty}\frac{\beta_n}{\rho_n}\left[-f(z_n,x^*)\right]<+\infty.$$
Consequently, under hypothesis C2 we obtain that $\lim_{n\to\infty}\inf \left[-f(z_n,x^*)\right]=0$, i.e.,
$$\lim_{n\to\infty}\sup f(z_n,x^*)=0.$$
 Similarly, from the boundedness of $\left\{u_n\right\}$ and (S5), we also get that 
$$\lim_{n\to\infty}\sup F(u_n,Ax^*)=0.$$
This completes the proof of Lemma \ref{lem2}.
\end{proof}
%%%%%%%%%%%%%%%%%%%%%%%%%%%%%%%%%%%%%%%%%%
%%%%%%%%%%%%%%%%%%%%%%%%%%%%%%%%%%%
Now, we prove the convergence of Algorithm \ref{alg1}.
\begin{theorem}\label{theo1}
The whole sequence $\left\{x_n\right\}$ generated by Algorithm \ref{alg1} converges weakly to some solution $x^\dagger$ of problem (SEP). 
Moreover, $x^\dagger=\lim_{n\to \infty}P_{\Omega}(x_n)$.
\end{theorem}
%%%%%%%%%%%%%%%%%%%%%%%%%%%%%%
\begin{proof}
Since $\left\{z_n\right\}$ is bounded, without loss of generality, we can assume that there exists a subsequence 
$\left\{z_m\right\}$ of $\left\{z_n\right\}$ converging weakly to $x^\dagger$ such that 
\begin{equation}\label{eq:10}
\lim_{n\to\infty}\sup f(z_n,x^*)=\lim_{m\to\infty}f(z_m,x^*).
\end{equation}
Since $C$ is closed and convex in a Hilbert space, $C$ is weakly closed in $H_1$. Thus, from $\left\{z_m\right\}\subset C$, we get that $x^\dagger\in C$.
Then, it follows from the weak upper semicontinuity of $f(.,x^*)$, relation (\ref{eq:10}) and Lemma \ref{lem2}(iii) that 
\begin{equation}\label{eq:11}
f(x^\dagger,x^*)\ge \lim_{m\to\infty}\sup f(z_m,x^*)=\lim_{m\to\infty} f(z_m,x^*)
= \lim_{n\to\infty}\sup f(z_n,x^*)=0.
\end{equation}
Since $x^*\in EP(f,C)$ and $x^\dagger\in C$, we have $f(x^*,x^\dagger)\ge 0$. Thus, from the pseudomonotonicity of $f$, we get that 
$f(x^\dagger,x^*)\le 0$. This together with relation (\ref{eq:11}) implies that $f(x^\dagger,x^*)=0$ and, using A4, that $x^\dagger\in EP(f,C)$. \\[.1in]
%%%%%%%%%%%%%%%%%%%%%%%%%%%%%%%%%%%%%%%%%%%%%5
Now we show that $Ax^\dagger\in EP(F,Q)$ and thus $x^\dagger\in \Omega$. Since 
$z_m\rightharpoonup x^\dagger$ and $||x_{m+1}-z_m||\to 0$, from Lemma \ref{lem2}(ii), 
we also have $x_{m+1}\rightharpoonup x^\dagger$, and thus $Ax_{m+1}\rightharpoonup 
Ax^\dagger$. Furthermore, also from Lemma \ref{lem2}(ii), we see that $||u_{m+1}-Ax_{m+1}||^2\to 0$ 
as $m\to\infty$, and thus $u_{m+1}\rightharpoonup Ax^\dagger$. The feasible set 
$Q$ being weakly closed and the subsequence $\left\{u_{m+1}\right\}$ being contained in 
$Q$, we obtain that $Ax^\dagger \in Q$. Arguing as in (\ref{eq:10}) and (\ref{eq:11}) but for the 
bifunction $F$, we also obtain that $Ax^\dagger \in EP(F,Q)$. \\[.1in]
%%%%%%%%%%%%%%%%%%%%%%%%%%%%%%%%%%%%%%%%%
Since $x^\dagger\in \Omega$ and Lemma \ref{lem2}(i), we can claim that the sequence $\left\{||x_n-x^\dagger||^2\right\}$ is convergent. 
Thus, from Lemma \ref{lem2}(ii), we also obtain the convergence of the sequence $\left\{||z_n-x^\dagger||^2\right\}$.
Now, we show the whole sequence $\left\{z_n\right\}$ converges weakly to $x^\dagger$. Indeed, 
assume that $\bar{x}$ is a weak cluster of the sequence $\left\{z_n\right\}$ such that $\bar{x}\ne x^\dagger$, i.e., there exists 
a subsequence $\left\{z_k\right\}$ of $\left\{z_n\right\}$ converging weakly to $\bar{x}$. It is obvious that $\bar{x}\in \Omega$ and thus that 
the sequence $\left\{||z_n-\bar{x}||^2\right\}$ is convergent. We have the following equality,
$$
2\left\langle z_n,\bar{x}-x^\dagger\right\rangle=||z_n-x^\dagger||^2-||z_n-\bar{x}||^2+||\bar{x}||^2-||x^\dagger||^2.
$$
Thus, the limit of the sequence $\left\{\left\langle z_n,\bar{x}-x^\dagger\right\rangle\right\}$ exists and is denoted by $l$, i.e.,
\begin{equation}\label{t1}
\lim_{n\to\infty}\left\langle z_n,\bar{x}-x^\dagger\right\rangle=l.
\end{equation}
Now, passing to the limit in (\ref{t1}) as $n=m\to\infty$ and after that $n=k\to\infty$, we obtain
$$
\left\langle x^\dagger,\bar{x}-x^\dagger\right\rangle=\lim_{m\to\infty}\left\langle z_m,\bar{x}-x^\dagger\right\rangle=
l=\lim_{k\to\infty}\left\langle z_k,\bar{x}-x^\dagger\right\rangle=\left\langle \bar{x},\bar{x}-x^\dagger\right\rangle.
$$
Hence, $||\bar{x}-x^\dagger||^2=0$ or $\bar{x}=x^\dagger$. This says that the whole sequence $\left\{z_n\right\}$ converges weakly to $x^\dagger$. 
Therefore, from Lemma \ref{lem2}(ii), we can conclude that the sequence $\left\{x_n\right\}$ converges weakly to $x^\dagger$. \\[.1in]
%%%%%%%%%%%%%%%%%%%%%%%%%%%%%%%%%%%%%%%%%%%%%
Finally, we prove $x^\dagger=\lim\limits_{n\to\infty}P_\Omega(x_n)$. Recalling the relation (\ref{eq7a})
\begin{equation}\label{h1}
||x_{n+1}-x^*||^2\le||x_n-x^*||^2+(1+\mu_n)\delta_n,~\forall x^*\in \Omega
\end{equation}
and substituting $x^*=P_\Omega(x_n)\in \Omega$ into (\ref{h1}), we obtain 
\begin{equation}\label{h2}
||x_{n+1}-P_\Omega(x_n)||^2\le||x_n-P_\Omega(x_n)||^2+(1+\mu_n)\delta_n.
\end{equation}
Since $\Omega$ is convex, we get from the definition of the metric projection that
$$
||x_{n+1}-P_\Omega(x_{n+1})||^2\le ||x_{n+1}-z||^2,~\forall z\in \Omega,
$$
which, with $z=P_\Omega(x_n)\in \Omega$, implies that 
\begin{equation}\label{h3}
||x_{n+1}-P_\Omega(x_{n+1})||^2\le||x_{n+1}-P_\Omega(x_n)||^2.
\end{equation}
Combining the relations (\ref{h2}) and (\ref{h3}), we come to the following estimate,
\begin{equation}\label{h4}
||x_{n+1}-P_\Omega(x_{n+1})||^2 \le ||x_n-P_\Omega(x_n)||^2+(1+\mu_n)\delta_n,
\end{equation}
or $a_{n+1}\le a_n+(1+\mu_n)\delta_n$ where $a_n=||x_n-P_\Omega(x_n)||^2$. Since $\sum\limits_{n=1}^\infty (1+\mu_n)\delta_n<+\infty$, 
from Lemma \ref{lem.tech1}, we see that the sequence $\left\{a_n\right\}$ converges as $n\to\infty$.\\[.1in] 
%%%%%%%%%%%%%%%%%%%%%%%%%%%%%%%%%%%%%%%%%%%%%%%%%%
For each $n\ge 1$, let $b_n=P_\Omega(x_n)$. Then, the sequence $\left\{b_n\right\}$ converges to some $b\in H_1$. Indeed, for each $n\ge 1$ 
and $p\ge 1$, it follows from Lemma \ref{lemPC}(ii), the definition of $b_n$, and the relation (\ref{h1}) that 
\begin{eqnarray*}
||b_{n+p}-b_n||^2&=&||P_\Omega(x_{n+p})-P_\Omega(x_n)||^2\\
&\le&||x_{n+p}-P_\Omega(x_n)||^2-||x_{n+p}-P_\Omega(x_{n+p})||^2\\
&\le&\left(||x_{n+p-1}-P_\Omega(x_n)||^2+(1+\mu_{n+p-1})\delta_{n+p-1}\right)-a_{n+p}\\
&\le&||x_{n+p-2}-P_\Omega(x_n)||^2+(1+\mu_{n+p-2})\delta_{n+p-2}\\
&&+(1+\mu_{n+p-1})\delta_{n+p-1}-a_{n+p}\\
&\le&\ldots\\
&\le&||x_n-P_\Omega(x_n)||^2+\sum_{t=n}^{n+p-1}(1+\mu_t)\delta_t -a_{n+p}\\
&=&(a_n-a_{n+p})+\sum_{t=n}^{n+p-1}(1+\mu_t)\delta_t. 
\end{eqnarray*}
Passing to the limit in the last inequality as $n,~p\to\infty$ and noting that $\sum\limits_{n=1}^\infty (1+\mu_n)\delta_n<+\infty$, we obtain 
$$\lim\limits_{n,~p\to\infty}||b_{n+p}-b_n||^2=0.$$
Thus, the sequence $\left\{b_n\right\}$ is a Cauchy sequence in $H_1$, i.e., there exists $b \in H_1$ such that $\lim\limits_{n\to\infty}b_n=b$. From $b_n=P_\Omega(x_n)$ 
and Lemma \ref{lemPC}(iii), we obtain 
\begin{equation}\label{h5}
\left\langle x^\dagger -b_n,x_n-b_n\right\rangle \le 0.
\end{equation}
Passing to the limit in (\ref{h5}) as $n\to\infty$, we find that $||x^\dagger-b||^2=\left\langle x^\dagger -b,x^\dagger-b\right\rangle \le 0$. Thus $b=x^\dagger$ 
or $x^\dagger=\lim\limits_{n\to\infty}P_\Omega(x_n)$. This finishes the proof.
\end{proof}
%%%%%%%%%%%%%%%%%%%%%%%%%%%%%%%%%%%%%
\section{Further remarks}\label{remarks}
%%%%%%%%%%%%%%%%%%%%%%%%%%%%%
In this section, we present several remarks regarding the assumptions of Theorem \ref{theo1} in the previous section and 
an extension of Algorithm \ref{alg1} in the case when $f$ and $F$ can be splitted into several bifunctions. We begin 
with assumption A3.
\begin{remark}\label{rem0}
Assumption A3 has been also considered by the authors in \cite{IS2003,SS2011,YMH2016}. This assumption is used to prove that 
the subgradient sequence $\left\{g_n\right\}$ is bounded when $\left\{z_n\right\}$ is bounded (similarly, with the 
sequence $\left\{w_n\right\}$ for bifunction $F$). We can assume directly as in \cite{SS2011} that the sequences 
$\left\{g_n\right\}$ and $\left\{w_n\right\}$ are bounded. However, from the proofs of Lemma \ref{lem2}(iii) and Theorem 
\ref{theo1}, we see that, without assumption A3, the result in this paper is still true if $f$ and $F$ are jointly weakly continuous 
on two open sets containing $C$ and $Q$, respectively, see, e.g. \cite[Proposition 4.3]{VSN2012}. 
\end{remark}
In the next remark, by an example, we show that assumption A4 is necessary in the formulation of Theorem \ref{theo1}.
\begin{remark}\label{rem1}
Algorithm \ref{alg1} converges under the assumption that $f$, $F$ satisfy paramonotone condition A4. The following simple example implies 
that, without this condition, the iterative sequence generated by the algorithm cannot converge (weakly) to any solution of the problem. Indeed, consider our 
problem with $C=Q=H_1=H_2=\Re^2$, $A=I$ and $f(x,y)=F(x,y)=x_1y_2-x_2y_1$ for all $x,~y\in \Re^2$. The problem has an unique solution 
$x^*=(0,0)^T$. Assumptions A1-A3 are automatically satisfied for $f$ and $F$. However, the hypothesis A4 does not hold. Indeed, we have that 
$f(y,x^*)=F(y,Ax^*)=0$ for all $y\in C=Q=\Re^2$ which cannot imply that $y\in EP(f,C)$ or $y\in EP(F,Q)$. Now, by some computation, from 
Algorithm \ref{alg1}, we obtain for each $n\ge 0$ and $x_n=(x_{1n},x_{2n})^T\in \Re^2$ that
\begin{eqnarray*}
&&y_n=(x_{1n}+\gamma_n x_{2n},x_{2n}-\gamma_n x_{1n})^T,\\ 
&&z_n= (x_{1n}+\mu_n \gamma_n x_{2n},x_{2n}-\mu_n \gamma_n x_{1n})^T.
\end{eqnarray*}
Thus, from the definiton of $x_{n+1}$, we obtain for each $n\ge 0$,
$$ x_{n+1}= ((1-\mu_n \alpha_n \gamma_n)x_{1n}+(\mu_n \gamma_n +\alpha_n) x_{2n},(-\mu_n \gamma_n - \alpha_n)x_{1n}+(1-\mu_n \alpha_n \gamma_n) x_{2n})^T. $$
By setting $a_n=1-\mu_n \alpha_n \gamma_n$ and $b_n=\mu_n \gamma_n +\alpha_n$, $x_{n+1}$ can be shortly rewritten as follows:
$$  x_{n+1}= (a_nx_{1n}+b_n x_{2n},-b_nx_{1n}+a_n x_{2n})^T.  $$
This implies that 
\begin{eqnarray*}
||x_{n+1}||^2&=&(a_nx_{1n}+b_n x_{2n})^2+(-b_nx_{1n}+a_n x_{2n})^2=(a_n^2+b_n^2)||x_n||^2.
\end{eqnarray*}
On the other hand, it follows from the definitions of $a_n$ and $b_n$ that 
$$ a_n^2+b_n^2=1+ \mu_n^2 \alpha_n^2 \gamma_n^2+\mu_n^2 \alpha_n^2 +\gamma_n^2>1.$$
Therefore $||x_{n+1}||^2>||x_n||^2$ for each $n\ge 0$, which implies, by the induction, that $||x_{n+1}||^2>||x_0||^2$. 
Thus, $\lim\limits_{n\to\infty}||x_{n+1}||^2>0$, provided that $x_0\ne 0$. This says that the sequence $\left\{x_n\right\}$ cannot 
converge to the solution $x^*=(0,0)^T$ of the problem. Since the weak convergence and strong convergence are the same in finite 
dimensional spaces, the sequence $\left\{x_n\right\}$ cannot converge weakly to the solution $x^*=(0,0)^T$.
\end{remark}
%%%%%%%%%%%%%%%%%%%%%%%%%%%%%
\begin{remark}
The convergence of Algorithm \ref{alg1} can be ensured under the assumption that the solution set $\Omega$ of problem (SEP) is nonempty. We remark here 
that, without this assumption, the algorithm 
can diverge. It is sufficient to consider our problem with $H_1=H_2=\Re^2$, $C=\left\{(x,0)\in H_1: x\ge 1\right\}$, 
$Q=\left\{(x,y)\in H_2: x\ge 1,~y\ge \frac{1}{\sqrt{x}}\right\}$, the operator $A=I$, and the two bifunctions $f(x,y)=\delta_C(y)-\delta_C(x)$ for all $x, y \in C$, 
and $F(x,y)=\delta_Q(y)-\delta_Q(x)$ for all $x,~y\in Q$, where $\delta_C$ and $\delta_Q$ are the indicator functions to $C$ and $Q$, respectively. It is easy 
to see that the solution set of problem (SEP) is $\Omega=C\cap Q=\emptyset$. Note that the projection of any point in $C$ onto $Q$ is always on the boundary 
of $Q$. Assume that at iteration $n$, we have $x_n=(x_{1n},0)^T\in C$ with $x_{1n}\ge 1$. From Algorithm \ref{alg1} and $A=I$, we see that $u_n=
P_Q(Ax_n)=P_Q(x_n)$. Since $u_n$ is on the boundary of $Q$, it is of the form $u_n=\left(u_{1n},\frac{1}{\sqrt{u_{1n}}}\right)\in Q$, 
where $u_{1n}$ is the unique solution of the strongly convex problem $\min_{t\ge 1}||a-x_n||^2$ with $a=(t,\frac{1}{\sqrt{t}})$, or 
\begin{equation}\label{opt}
\min_{t\ge 1}\left\{h(t)= (t-x_{1n})^2+\frac{1}{t}\right\}.
\end{equation}
We have that $h'(t)=2(t-x_{1n})-\frac{1}{t^2}$. By a straightforward computation, we see that $h'(x_{1n}+\frac{1}{4x_{1n}^2})<0$. Thus, the unique 
optimal solution $u_{1n}$ of problem (\ref{opt}) must satisfy the inequality  $u_{1n}>x_{1n}+\frac{1}{4x_{1n}^2}$. Since $0\in \partial f_2(x,x)$ 
and $0\in \partial F_2(x,x)$ for all $x$, we can choose $w_n=g_n=0\in \Re^2$. Moreover, we can take $\mu_n=\frac{1}{2}\in (0,1)=(0,\frac{1}{||A||^2})$. 
Thus, from Algorithm \ref{alg1}, we obtain that $y_n=P_Q(u_n)=u_n$, $z_n=P_C\left(\frac{x_n+y_n}{2}\right)$ 
and $x_{n+1}=P_C(z_n)=P_C\left(\frac{x_n+y_n}{2}\right)=P_C\left(\frac{x_n+u_n}{2}\right)=(\frac{u_{1n}+x_{1n}}{2},0)^T$. This together with 
the inequality $u_{1n}>x_{1n}+\frac{1}{4x_{1n}^2}$ implies that 
$$ x_{1,n+1}=\frac{u_{1n}+x_{1n}}{2}>x_{1n}+\frac{1}{8x_{1n}^2},~\forall n\ge 0. $$
Thus, it is not difficult to see that $x_{1n}\to +\infty$. Hence $||x_n||=|x_{1n}|\to +\infty$ as $n\to\infty$. This says that the sequence $\left\{x_n\right\}$ 
generated by Algorithm \ref{alg1} diverges.
\end{remark}
%%%%%%%%%%%%%%%%%%%%%%%%%
\begin{remark}\label{rem3}
Algorithm \ref{alg1} can be extended to the case when $f=\sum_{i=1}^Nf_i$ and $F=\sum_{j=1}^M F_j$. In that case, the parallel projection algorithm is given by 
\begin{equation}\label{PPSM}
\begin{cases}
u_n=P_Q(Ax_n),~w_{n}^j\in \partial_{\epsilon_n}F_j(u_n,.)(u_n),\\
\gamma_n=\frac{\beta_n}{\max\left\{\rho_n,||w^1_{n}||,\ldots,||w^M_{n}||\right\}},~y^j_n=P_Q(u_n-\gamma_n w^j_{n}),\\
y_n=\frac{1}{M}\sum_{j=1}^M y_n^j,~z_{n}=P_C\left(x_n+\mu_n A^*(y_n-Ax_n)\right),\\
g_n^i\in \partial_{\epsilon_n}f_i(z_n,.)(z_n),~\alpha_{n}=\frac{\beta_n}{\max\left\{\rho_n,||g_n^1||,\ldots,||g_n^N||\right\}},\\
x_{n}^i=P_C(z_n-\alpha_{n} g_n^i),x_{n+1}=\frac{1}{N}\sum_{i=1}^N x_n^i.
\end{cases}
\end{equation}
Under the assumptions as in Theorem \ref{alg1}, the sequence $\left\{x_n\right\}$ generated by (\ref{PPSM}) converges weakly 
to some solution $x^\dagger$ of problem (SEP). Moreover, $x^\dagger=\lim_{n\to\infty}P_{\Omega}(x_n)$.
\end{remark}
%%%%%%%%%%%%%%%%%%%%%%%%%%
\begin{remark}\label{rem4}
Algorithm \ref{alg1} is performed under the previous knowledge of the norm of operator $A$. An open question, is then to design an algorithm 
which can be used without the prior knowledge of the operator norm as, for instance, in \cite[Algorithm 4.1]{H2017GCOM} for problem 
(SEP) or in \cite{MT2014} for the split feasibility problem.
\end{remark}
%%%%%%%%%%%%%%%%%%%%%%%%%%
\section{Split common equilibrium problems}\label{extension}
%%%%%%%%%%%%%%%%%%%%%%%%%%
This section deals with an extension of Algorithm \ref{alg1} for solving the split common equilibrium problem (SCEP) considered 
in \cite{CGR2012,H2012,H2016}. This problem is stated as follows:
\begin{equation*}\label{SCEP}
\begin{cases}
\mbox{Find}~x^*\in C~\mbox{such that} ~f_i(x^*,y)\ge 0,~\forall y\in C,~i\in I,\\
\mbox{and}~u^*=Ax^*\in Q~\mbox{solves} ~F_j(u^*,u)\ge 0,~\forall u\in Q,~j\in J,
\end{cases}
\eqno{\rm (SCEP)}
\end{equation*}
where $I=\left\{1,2,\ldots,N\right\},~J=\left\{1,2,\ldots,M\right\}$; $C$ and $Q$ are two nonempty closed convex subsets of 
two real Hilbert spaces $H_1,~H_2$, respectively; 
$A:H_1\to H_2$ is a bounded linear operator; and $f_i:C\times C\to \Re$ and $F_j:Q\times Q\to \Re$ are bifunctions with 
$f_i(x,x)=0$ for all $x\in C$ and $F_j(u,u)=0$ for all $u\in Q$. We denote here by $\Omega$ the solution set of problem (SCEP) 
and assume that it is nonempty. It is well known that problem (SCEP) contains properly many split-like problems, see, e.g., 
\cite{CGR2012}. For solving problem (SCEP), He \cite{H2012} used the resolvent of a bifunction (the proximal point method) 
to propose a weakly convergent parallel algorithm \cite[algorithm (3.2)]{H2012} in the case $N>1$ and $M=1$. 
In a different direction, the author in \cite{H2016} has additionally incorporated in the previous algorithm the extended extragradient method and has proposed 
two weakly and strongly convergent parallel algorithms. In this section, as an extension of Algorithm \ref{alg1}, we present a 
different algorithm, which only uses the projections to design. \\[.1in]
%%%%%%%%%%%%%%%%%%%%%%%%%%%%%%%%%%
In order to solve problem (SCEP), we also assume that for each $i\in I$ and $j\in J$, the 
bifunctions $f_i$ and $F_j$ have the same properties as $f$ and $F$ in Section \ref{main}. The algorithm is designed as follows:

%\noindent\rule{12.1cm}{0.4pt} 
\begin{algorithm}[Parallel algorithm for SCEPs].\label{alg2}\\
\noindent\rule{12.6cm}{0.4pt}\\
\textbf{Initialization:} Choose $x_0\in C$ and the parameter sequences $\left\{\rho_n\right\}$, $\left\{\beta_n\right\}$, $\left\{\epsilon_n\right\}$, 
$\left\{\mu_n\right\}$ such that condition C1-C3 above hold. Moreover, consider additionally the sequences $\left\{\theta_n^j\right\},~\left\{\tau_n^i\right\}
\subset [b,c]\subset (0,1)$ such that $\sum_{j\in J}\theta_n^j=\sum_{i\in I}\tau_n^i=1$.\\
\noindent\rule{12.6cm}{0.4pt}\\
\textbf{Iterative Steps:} Assume that $x_n\in C$ is known, calculate $x_{n+1}$ as follows:

\textbf{Step 1.} Select $w^j_{n}\in \partial_{\epsilon_n}F_j(u_n,.)(u_n)$ where $u_n=P_Q(Ax_n)$, and compute 
$$\gamma^j_{n}=\frac{\beta_n}{\max\left\{\rho_n,||w^j_{n}||\right\}},\qquad y^j_n=P_Q(u_n-\gamma^j_{n} w^j_{n}),\qquad y_n=\sum_{j\in J}\theta^j_n y_n^j.$$

\textbf{Step 2.} Compute $ z_{n}=P_C\left(x_n+\mu_n A^*(y_n-Ax_n)\right)$.

\textbf{Step 3.} Select $g^i_n\in \partial_{\epsilon_n}f_i(z_n,.)(z_n)$ and compute 
$$\alpha^i_n=\frac{\beta_n}{\max\left\{\rho_n,||g^i_n||\right\}},\qquad x^i_{n}=P_C(z_n-\alpha^i_n g^i_n),\qquad x_{n+1}=\sum_{i\in I}\tau^i_n x_n^i.$$

Set $n:=n+1$ and go back to \textbf{Step 1}.
\end{algorithm}
\noindent\rule{12.6cm}{0.4pt}\\[.1in]
%%%%%%%%%%%%%%%%%%%%%%%%%%%%%%%%%%%%
We omit here the proof of convergence of Algorithm \ref{alg2}. In fact, it is easy to obtain it by repeating the proofs in the previous section. 
We have the following result.
\begin{theorem}\label{theo2}
The sequence $\left\{x_n\right\}$ generated by Algorithm \ref{alg2} converges weakly to some solution $x^\dagger$ of problem (SCEP). Moreover, 
$x^\dagger=\lim_{n\to\infty}P_\Omega(x_n)$.
\end{theorem}
%%%%%%%%%%%%%%%%%%%%%%%%%%%%%%%%%%
\section{Computational experiments}\label{example}
% %%%%%%%%%%%%%%%%%%%%%%%%%%%%%%%%%%%%%
This section presents several experiments to illustrate the numerical behavior of Algorithm \ref{alg1} (shortly, PM) and also to compare it with the behaviors of other 
well known algorithms. The test problem here can be considered as an extension of the Nash-Cournot oligopolistic equilibrium model in \cite{CKK2004,FP2002} 
to the split equilibrium model in \cite{H2017GCOM}. More precisely, the problem is for $H_1=\Re^m$ and $H_2=\Re^k$.  The bifunction $f$ on $H_1$ is 
of the form 
$$ f(x,y)= \left\langle \bar{M} x+\bar{N}y+p,y-x\right\rangle,$$
where $p$ is a vector in $\Re^m$ and $\bar{M},~\bar{N}$ are two matrices of order $m$ such that $\bar{N}$ is symmetric positive semidefinite and $\bar{N}-\bar{M}$ is negative 
semidefinite. The bounded linear operator $A:\Re^m\to \Re^k$ is defined by a matrix of size $k\times m$. All the entries of $A$ are generated randomly 
(and uniformly) in $[-10,10]$. The bifunction $F$ also has the following 
form
$$ F(x,y)=\left\langle Mx+Ny+q,y-x\right\rangle$$
where $q$ is a vector in $\Re^k$ and $M,~N$ are two matrices of order $k$ such that $N$ is symmetric positive semidefinite and $N-M$ is negative 
semidefinite. Two feasible sets respectively are $C=[-1,5]^m$ and $Q=[-2,5]^k$. In the purpose that the solution set of the problem is nonempty and 
that all the algorithms can work, the two vectors $p$ and $q$ are chosen as the two zero vectors in $\Re^m$ and $\Re^k$, respectively.  The matrices 
$\bar{M}$ and $\bar{N}$ are generated randomly to satisfy the conditions\footnote{Choose randomly $\lambda_{1k}\in [-10,0],~
\lambda_{2k}\in [1,10]$ for all $k=1,\ldots,m$. Set $\widehat{Q}_1$, $\widehat{Q}_2$ as two diagonal matrixes with eigenvalues 
$\left\{\lambda_{1k}\right\}_{k=1}^m$ and $\left\{\lambda_{2k}\right\}_{k=1}^m$, respectively. Then, we consider a positive semidefinite matrix 
$\bar{N}$ and a negative semidefinite matrix $T$ by using full random orthogonal matrixes with $\widehat{Q}_2$ and $\widehat{Q}_1$, respectively. 
Finally, set $\bar{M}=\bar{N}-T$} (the matrices $M, ~N$ are also generated randomly at this way).\\[.1in]
%%%%%%%%%%%%%%%%%%%%%%%%%%%%%%%%%%%%%%%%%%%%%%%%%%%%%%%%%%%%%%%%%%%%%%%
This section is divided into two parts: Subsection \ref{exper1} studies the numerical behavior of Algorithm \ref{alg1}, 
while Subsection \ref{exper2} 
reports several results in comparing Algorithm \ref{alg1} with other algorithms, namely the Extragradient-Proximal Method (EGPM) in \cite[Algorithm 1]{H2016}; 
the Hybrid Extragradient- Proximal Method (HEGPM) in \cite[Algorithm 2]{H2016}; the Projected Subgradient-Proximal Method (PSPM) in 
\cite[Algorithm 3.1]{H2017GCOM}; and the Modified Projected Subgradient-Proximal Method (MPSPM) in \cite[Algorithm 4.1]{H2017GCOM}. 
The solution of the considered problem 
is $x^*=0$ and it is easy to see that conditions A1-A4 are satisfied. Thus, also as in \cite{H2017GCOM}, all the algorithms can be applied. 
We have used the sequence 
$D_n=||x_n-x^*||^2$, $n=0,1,\ldots$ to study the convergence of all the algorithms. The starting point is 
$x_0=(1,1,\ldots,1)^T\in H_1$. The convergence of $D_n$ to $0$ implies that the sequence $\left\{x_n\right\}$ generated by each algorithm 
converges to the solution $x^*$ of the problem.\\[.1in]
%%%%%%%%%%%%%%%%%%%%%%%%%%%%%%%%%%%%%%
All the projections and the optimization problems are solved effectively by using the function \textit{quadprog} in the Matlab 7.0 Optimization 
Toolbox. All the programs are written in Matlab and computed on a PC Desktop Intel(R) Core(TM) i5-3210M CPU @ 2.50 GHz, RAM 2.00 GB. 
\subsection{Numerical behavior of Algorithm \ref{alg1}}\label{exper1}
In this part, the four matrices $M$, $N$, $\bar{M}$ and $\bar{N}$ are generated ramdomly. In this case, it is not easy 
to implement the four algorithms EGPM, HEGPM, PSPM, MPSPM because they use the resolvent mapping which in general is difficult to compute. 
Then, we only illustrate the numerical behavior of Algorithm \ref{alg1}. The six sequences of $\left\{\beta_n\right\}$ 
are taken as $\beta_n=\frac{1}{(n+1)^s},~s\in \left\{1;~0.9;~0.8;~0.7;~0.6;~0.51\right\}$. Other parameters are $\epsilon_n=0$, $\rho_n=1$, 
$\mu_n=\frac{1}{||A||^2}$. Figures \ref{fig1} - \ref{fig4} show the behavior of $\left\{D_n\right\}$ generated by Algorithm \ref{alg1} for different pairs of $(m,k)$. 
In each figure, the $y$-axis represents the value of $D_n$ while the $x$-axis is for the execution time elapsed in second. 
In view of these figures, we see that the rate of convergence of Algorithm \ref{alg1} depends strictly on the rate of convergence of the sequence 
$\left\{\beta_n\right\}$. Algorithm \ref{alg1} in general works well for the sequences $\beta_n=\frac{1}{(n+1)^s}$ with $s\in \left\{0.8,~0.7,~0.6,~0.51\right\}$, 
and it is especially noted that in all the cases the new algorithm works badly for the natural sequence 
$\beta_n=\frac{1}{n+1}$. 
%%%%%%%%%%%%%%%%%%%%%%%%%%%%%%%5
\begin{figure}[!ht]
\centering
\includegraphics[height=5cm,width=6cm]{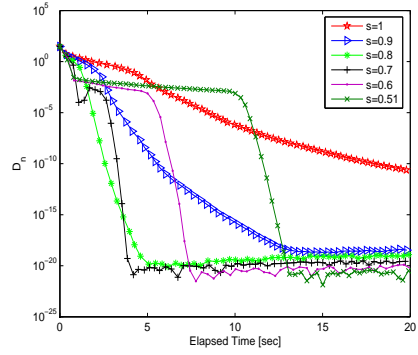}
\caption{Algorithm \ref{alg1} for $(m,k)=(30,20)$ and different sequences of $\beta_n$. The number of iterations is 
360,   353,   339,   360,   355,   376, respectively.}\label{fig1}
\end{figure}
%%%%%%%%%%%%%%%%%%%%%%%%%%%%%%%%%%%
\begin{figure}[!ht]
\centering
\includegraphics[height=5cm,width=6cm]{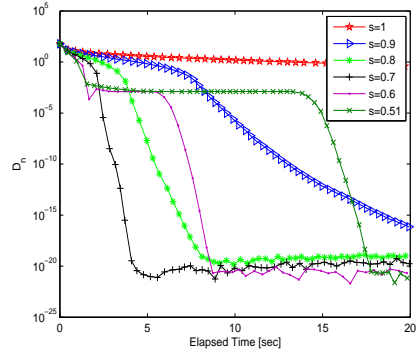}
\caption{Algorithm \ref{alg1} for $(m,k)=(60,40)$ and different sequences of $\beta_n$. The number of iterations is  
258,   333,   336,   326,   291,   293, respectively.}\label{fig2}
\end{figure}
%%%%%%%%%%%%%%%%%%%%%%%%%%%%%%%%%%%
%%%%%%%%%%%%%%%%%%%%%%%%%%%%%%%%%%%
\begin{figure}[!ht]
\centering
\includegraphics[height=5cm,width=6cm]{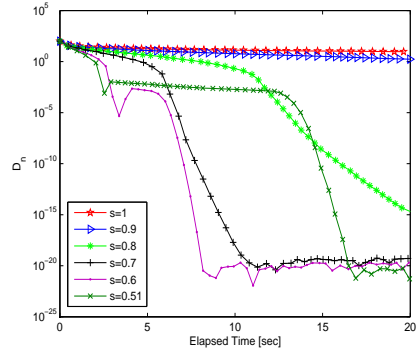}
\caption{Algorithm \ref{alg1} for $(m,k)=(100,50)$ and different sequences of $\beta_n$. The number of iterations is  
215,   236,   283,   280,   321,   290, respectively.}\label{fig3}
\end{figure}
\begin{figure}[!ht]
\centering
\includegraphics[height=5cm,width=6cm]{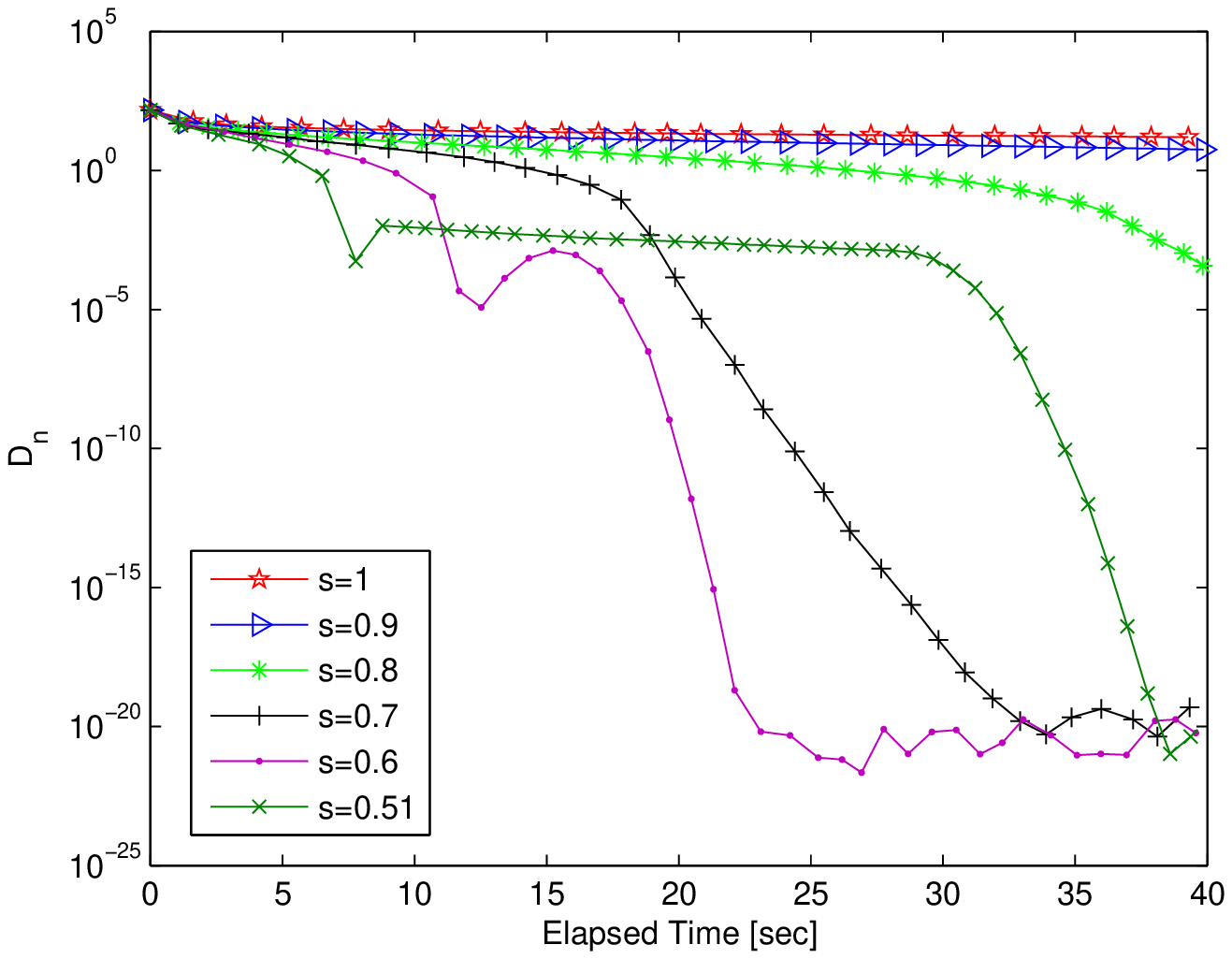}
\caption{Algorithm \ref{alg1} for $(m,k)=(150,100)$ and different sequences of $\beta_n$. The number of iterations is  
161,   188,   219,   209,   245,   264, respectively.}\label{fig4}
\end{figure}
%%%%%%%%%%%%%%%%%%%%%%%%%%%%%%%%%%%
\subsection{Comparison of Algorithm \ref{alg1} with other algorithms}\label{exper2}
Four aforementioned algorithms EGPM, HEGPM, PSPM and MPSPM have been designed from the resolvent $T_r^F$ of the bifunction $F$. 
In order to compute easily the value of the resolvent mapping $T_r^F$, we have chosen $M=N$. In that case, the resolvent $T^F_r$ 
of $F$ coincides with the proximal mapping of the function $g(x)=\left\langle Mx,x\right\rangle$ for $r>0$, i.e., $T_r^F(u)={\rm prox}_{rg}(u)$, 
where 
$$ {\rm prox}_{rg}(u)=\arg\min\left\{ g(v)+\frac{1}{r}||v-u||^2:v\in Q\right\}. $$
The mapping ${\rm prox}_{rg}(u)$ can be effectively computed by using the Optimization Toolbox in Matlab. Moreover, it is emphasized that 
although the conditions of convergence of the four compared algorithms in general are different to the ones of Algorithm \ref{alg1}, we still wish 
to present a numerical comparison between them. For implementing algorithms EGPM and HEGPM \cite{H2016}, we need two Lipschitz-type 
constants $c_1 $ and $c_2$ of $f$ (they are $c_1=c_2=||\bar{M}-\bar{N}||/2$). The parameters have been chosen in all the experiments as follows:\\[.1in]

(i) $\lambda=\frac{1}{5c_1},~\mu_n=\mu=\frac{1}{||A||^2}$ for EGPM, HEGPM, PSPM and Algorithm \ref{alg1} (PM).\\[.1in]

(ii) $\epsilon_n=0,~\rho_n=1,~\beta_n=\frac{1}{(n+1)^{0.7}}$ for PSPM, MPSPM and Algorithms \ref{alg1}.\\[.1in]

(iii) $\nu_n=3$ for MPSPM and $r_n=1$ for EGPM, HEGPM, PSPM, MPSPM.\\[.1in]
Figures \ref{fig5} - \ref{fig8} describe the behavior of the sequence $\left\{D_n\right\}$ generated by the algorithms. In view of this, we see that 
the proposed algorithm has competitive advantages over existing algorithms. It is also seen that the obtained error from Algorithm 
\ref{alg1} is better than the one from other algorithms.
\begin{figure}[!ht]
\centering
\includegraphics[height=5cm,width=6cm]{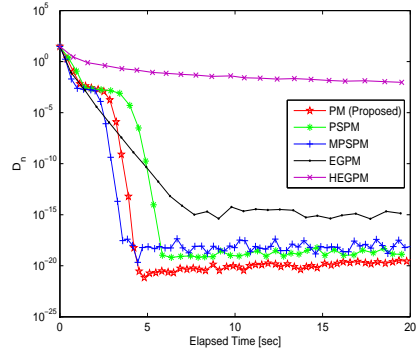}
\caption{Experiment for the algorithms with $(m,k)=(30,20)$. The number of iterations is  334,   240,   379,   168,   130, respectively.}\label{fig5}
\end{figure}
\begin{figure}[!ht]
\centering
\includegraphics[height=5cm,width=6cm]{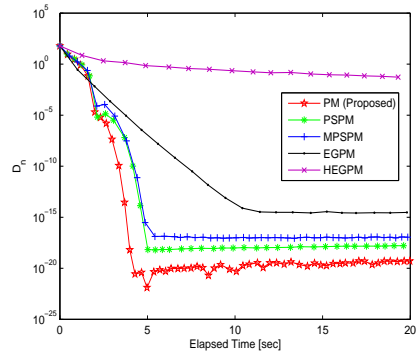}
\caption{Experiment for the algorithms with $(m,k)=(60,40)$. The number of iterations is  326,   221,   292,   129,   108, respectively.}\label{fig6}
\end{figure}

\begin{figure}[!ht]
\centering
\includegraphics[height=5cm,width=6cm]{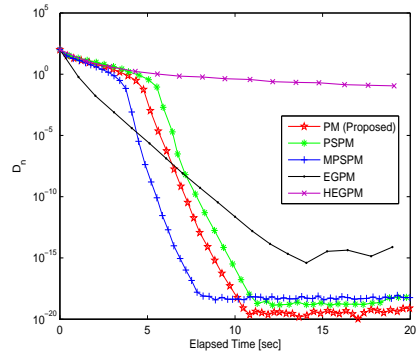}
\caption{Experiment for the algorithms with $(m,k)=(100,50)$. The number of iterations is  308,   250,   356,   114,    89, respectively.}\label{fig7}
\end{figure}
\begin{figure}[!ht]
\centering
\includegraphics[height=5cm,width=6cm]{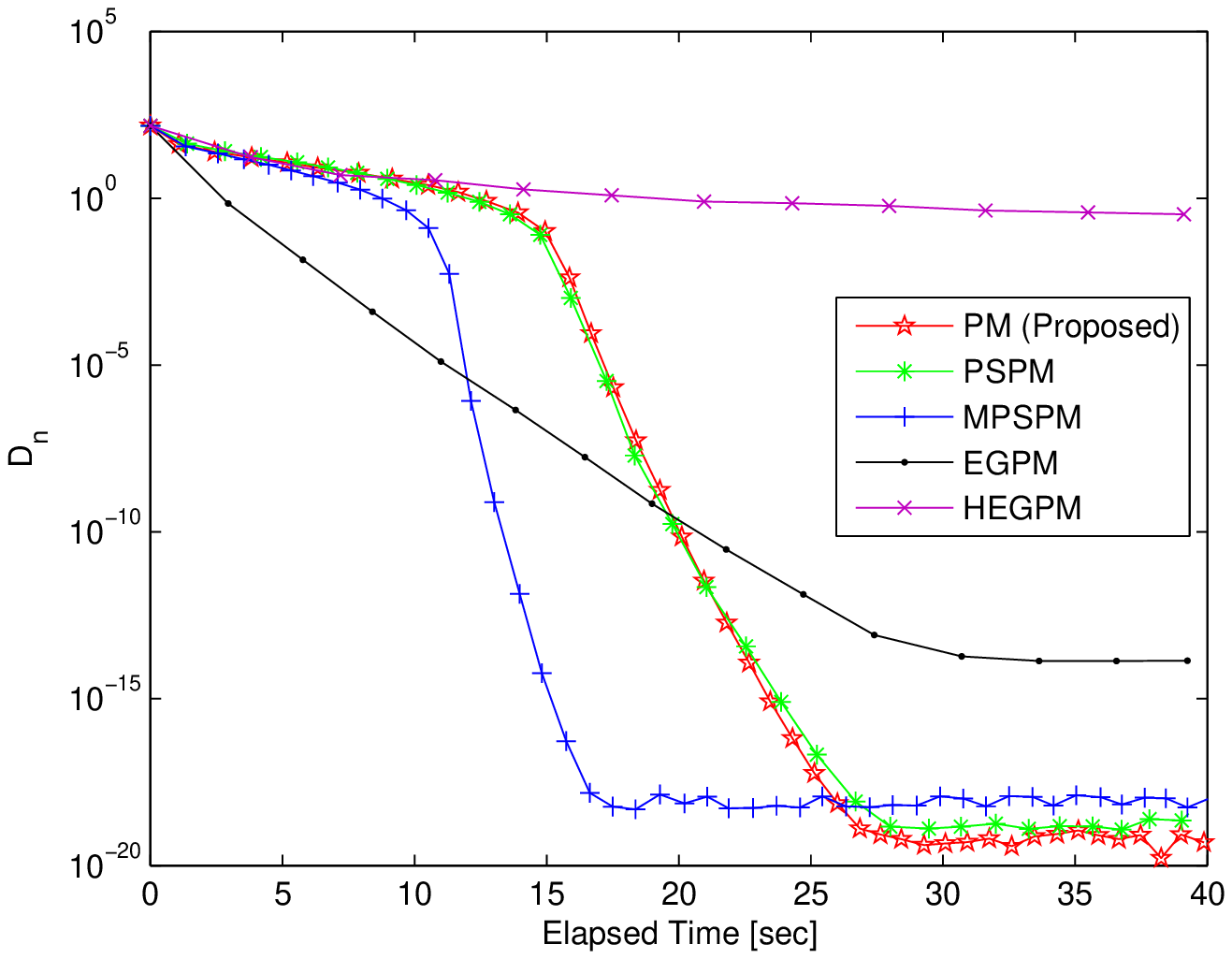}
\caption{Experiment for the algorithms with $(m,k)=(150,100)$. The number of iterations is  254,   192,   271,    87,    69, respectively.}\label{fig8}
\end{figure}
%%%%%%%%%%%%%%%%%%%%%%%%%%%%%%%%%%%%
\section{Conclusions}
The paper has considered a class of split inverse problems involving equilibrium problems in Hilbert spaces, so-called briefly the split equilibrium problem. 
This problem unifies in a simple form various previously known split-type problems. A new algorithm, which only uses the projections to design, has been 
proposed for approximating the solutions. A theorem of weak convergence has been proved under suitable conditions. The convergent conditions are also 
discussed, and as be seen, they are almost necessary in the formulation of the convergence theorem. Several extensions of the resulting algorithm to the split 
common equilibrium problem have been also presented in the paper. The numerical behavior of the new algorithm is studied by reporting some numerical 
experiments. In particular, it is seen that the proposed algorithm also has competitive advantages over existing methods.

\section*{Acknowledgments} 
The author would like to thank the Associate Editor and the two anonymous referees 
for their valuable comments and suggestions which helped us very much in improving the original version of this paper.  
This work is supported by Vietnam 
National Foundation for Science and Technology Development (NAFOSTED) under the project: 101.01-2017.315


\begin{thebibliography}{99}
\bibitem{AM14} 
\newblock P. N. Anh and L. D. Muu,
\newblock A hybrid subgradient algorithm for nonexpansive mappings and equilibrium problems,
\newblock \emph{Optim. Lett.}, \textbf{8} (2014), 727-738.

\bibitem{BO1994}
\newblock E. Blum and W. Oettli,
\newblock From optimization and variational inequalities to equilibrium problems,
\newblock \emph{Math. Program.}, \textbf{63} (1994), 123--145.


\bibitem{B2002} 
\newblock C. Byrne,
\newblock Iterative oblique projection onto convex sets and the split feasibility problems,
\newblock \emph{Inverse Prob.}, \textbf{18} (2002), 441-453.


\bibitem{B2004} 
\newblock C. Byrne,
\newblock A unified treatment of some iterative algorithms in signal processing and image reconstruction,
\newblock \emph{Inverse Prob.}, \textbf{20} (2004), 103-120.


\bibitem{CBMT2006} 
\newblock Y. Censor, T. Bortfeld, B. Martin, and A. Trofimov,
\newblock A unified approach for inversion problems in intensitymodulated radiation therapy,
\newblock \emph{Phys. Med. Biol.}, \textbf{51} (2006), 2353-2365.


\bibitem{CE1994}
\newblock Y. Censor and T. Elving,
\newblock A multiprojections algorithm using Bregman projections in a product spaces,
\newblock \emph{Numer. Algor.}, \textbf{81} (1994), 221-239.

\bibitem{CGR2012}
\newblock Y. Censor, A. Gibali and S. Reich,
\newblock Algorithms for the split variational inequality problem,
\newblock \emph{Numer. Algor.}, \textbf{59} (2012), 301-323.


\bibitem{CS2008} 
\newblock Y. Censor and A. Segalh,
\newblock Iterative projection methods in biomedical inverse problems. In: Censor Y, Jiang M, Louis AK (eds) Mathematical methods in biomedical imaging and intensity-modulated therapy, 
\newblock \emph{IMRT, Edizioni della Norale, Pisa}, (2008), 65-96.

\bibitem{CWWW2015}
\newblock S. Chang, L. Wang, X. R. Wang and G. Wang,
\newblock General split equality equilibrium problems with application to split optimization problems,
\newblock \emph{J. Optim. Theory Appl.}, \textbf{166} (2015), 377-390.

\bibitem{CH2005}
\newblock P. L. Combettes and S. A. Hirstoaga,
\newblock Equilibrium programming in Hilbert spaces,
\newblock \emph{J. Nonlinear Convex Anal.} \textbf{6} (2005), 117--136.

\bibitem{CKK2004} 
\newblock J. Contreras, M. Klusch, J. B. Krawczyk,
\newblock Numerical solution to Nash-Cournot equilibria in coupled constraint electricity markets,
\newblock \emph{EEE Trans. Power. Syst.} \textbf{19} (2004), 195-206.


\bibitem{DMSK2015} 
\newblock J. Deepho, J. Martnez-Moreno, K. Sitthithakerngkiet and P. Kumam,
\newblock Convergence analysis of hybrid projection with Cesaro mean method for the split equilibrium and general system of finite variational inequalities,
\newblock \emph{J. Comput. Appl. Math.}  \textbf{318}, (2017) 658-673.

\bibitem{DKK2014} 
\newblock J. Deepho, W. Kumam and P. Kumam,
\newblock A new hybrid projection algorithm for solving the split generalized equilibrium problems and the system of variational inequality problems,
\newblock \emph{J. Math. Model. Algor.} \textbf{13}, (2014) 405-423.

\bibitem{DSA17} 
\newblock B. V. Dinh, D. X. Son, T. V. Anh,
\newblock Extragradient-proximal methods for split equilibrium and fixed point problems in Hilbert spaces,
\newblock \emph{Vietnam J. Math.} \textbf{45},(2017) 651-668.

\bibitem{FP2002} 
\newblock F. Facchinei and J. S. Pang,
\newblock \emph{Finite-Dimensional Variational Inequalities and Complementarity Problems},
\newblock Springer, Berlin, 2002.

\bibitem{FA1997} 
\newblock S. D. Flam and A. S. Antipin,
\newblock Equilibrium programming and proximal-like algorithms,
\newblock \emph{Math. Program.} \textbf{78}, (1997) 29-41.

\bibitem{GR1984} 
\newblock K. Goebel and S. Reich,
\newblock \emph{Uniform Convexity, Hyperbolic Geometry, and Nonexpansive Mappings},
\newblock Marcel Dekker, New York and Basel, 1984.

\bibitem{H2012} 
\newblock Z. He,
\newblock The split equilibrium problems and its convergence algorithms,
\newblock \emph{J. Inequal. Appl.} (2012), 2012.

\bibitem{H2017OPTL} 
\newblock D. V. Hieu,
\newblock Projected subgradient algorithms on systems of equilibrium problems,
\newblock \emph{Optim. Lett.} \textbf{12}, (2018) 551-566.

\bibitem{H2016}
\newblock D. V. Hieu,
\newblock \doititle{Parallel extragradient-proximal methods for split equilibrium problems},
\newblock \emph{Math. Model. Anal.}, \textbf{21} (2016), 478--501.

\bibitem{H2017GCOM} 
\newblock D. V. Hieu,
\newblock Two hybrid algorithms for solving split equilibrium problems,
\newblock \emph{Inter. J. Comput. Math.} \textbf{95}, (2018) 561-583.

\bibitem{HM2017} 
\newblock D. V. Hieu and A. Moudafi,
\newblock A barycentric projected-subgradient algorithm for equilibrium problems,
\newblock \emph{J. Nonlinear Var. Anal.} \textbf{1}, (2017) 43-59.

\bibitem{HS18}
\newblock D. V. Hieu and J. J. Strodiot,
\newblock Strong convergence theorems for equilibrium problems and fixed point problems in Banach spaces,
\newblock \emph{J. Fixed Point Theory Appl.  } \textbf{20:131}, (2018) 1-32.


\bibitem{H18}
\newblock D. V. Hieu,
\newblock An inertial-like proximal algorithm for equilibrium problems,
\newblock \emph{Math. Meth. Oper. Res.} \textbf{88}, (2018) 399-415.

\bibitem{HCB18}
\newblock D. V. Hieu, Y. J. Cho and Y-B. Xiao,
\newblock Modified extragradient algorithms for solving equilibrium problems,
\newblock \emph{Optimization} \textbf{67}, (2018) 2003-2029.

\bibitem{H1989} 
\newblock N. E. Hurt,
\newblock \emph{Phase Retrieval and Zero Crossings: Mathematical Methods in Image Reconstruction},
\newblock  Kluwer Academic, Dordrecht, The Netherlands, 1989.

\bibitem{IS2003} 
\newblock A. N. Iusem and W. Sosa,
\newblock Iterative algorithms for equilibrium problems,
\newblock \emph{Optimization} \textbf{52}, (2003) 301-316.

\bibitem{KR2013} 
\newblock K. R. Kazmi and S. H. Rizvi,
\newblock Iterative approximation of a common solution of a split equilibrium problem, a variational inequality problem and a fixed point problem,
\newblock \emph{J. Egyptian Math. Society} \textbf{21}, (2013) 44-51.

\bibitem{M1999} 
\newblock A. Moudafi,
\newblock Proximal point algorithm extended to equilibrum problem,
\newblock \emph{J. Nat. Geometry} \textbf{15}, (1999) 91-100.

\bibitem{M2011} 
\newblock A. Moudafi,
\newblock Split monotone variational inclusions,
\newblock \emph{J. Optim. Theory Appl.} \textbf{150}, (2011) 275-283.

\bibitem{MT2014} 
\newblock A. Moudafi and B. S. Thakur,
\newblock Solving proximal split feasibility problems without prior knowledge of operator norms,
\newblock \emph{Optim. Lett.} \textbf{8}, (2014) 2099-2110.

\bibitem{M2013} 
\newblock A. Moudafi and E. Al-Shemas,
\newblock Simultaneously iterative methods for split equality problem,
\newblock \emph{Trans. Math. Program. Appl.} \textbf{1}, (2013) 1-11.

\bibitem{M2013a} 
\newblock A. Moudafi,
\newblock A relaxed alternating CQ algorithm for convex feasibility problems,
\newblock \emph{Nonlinear Anal. TMA} \textbf{79}, (2013) 117-121.

\bibitem{MO1992}
\newblock L. D. Muu and W. Oettli,
\newblock \doititle{Convergence of an adative penalty scheme for finding constrained equilibria},
\newblock \emph{Nonlinear Anal. TMA,} \textbf{18} (1992), 1159--1166.

\bibitem{QMH2008}
\newblock T. D. Quoc, L. D. Muu and N. V. Hien,
\newblock \doititle{Extragradient algorithms extended to equilibrium problems},
\newblock \emph{Optimization}, \textbf{57} (2008), 749--776.

\bibitem{SS2011} 
\newblock P. Santos and S. Scheimberg,
\newblock An inexact subgradient algorithm for equilibrium problems,
\newblock \emph{Comput. Appl. Math.} \textbf{30}, (2011) 91-107.

\bibitem{S1987}  
\newblock H. Stark,
\newblock \emph{Image Recovery: Theory and Applications},
\newblock Academic Press, Orlando, FL, 1987.

\bibitem{VSN2012}
\newblock P. T. Vuong, J. J. Strodiot and V. H. Nguyen,
\newblock \doititle{Extragradient methods and linesearch algorithms for solving Ky Fan inequalities and fixed point problems},
\newblock \emph{J. Optim. Theory Appl.}, \textbf{155} (2012), 605--627.

\bibitem{X2004} 
\newblock H. K. Xu,
\newblock Viscosity approximation methods for nonexpansive mappings,
\newblock \emph{Math. Anal. Appl.} \textbf{298}, (2004) 279-291.

\bibitem{YMH2016} 
\newblock L. H. Yen, L. D. Muu and N. T. T. Huyen,
\newblock An algorithm for a class of split feasibility problems: application to a model in electricity production,
\newblock \emph{Math. Meth. Oper. Res.} \textbf{84}, (2016) 549-565.
\end{thebibliography}
\end{document}